%% file: version1.tex
\documentclass[review]{siamart1116}

\input{ex_shared}

\ifpdf
\hypersetup{
  pdftitle={On passivity of fractional order systems},
  pdfauthor={Mohsen Rakhshan, Vijay Gupta, Bill Goodwine}
}
\fi

\newsiamthm{assumption}{Assumption} 
\newsiamthm{remark}{Remark} 

\begin{document}
\nolinenumbers
\maketitle

\begin{abstract}
We generalize notions of passivity and dissipativity to fractional order systems. Similar to integer order systems, we show that the proposed definitions generate analogous stability and compositionality properties for fractional order systems as well. We also study problem of passivating a fractional order system through a feedback controller. Numerical examples are presented to illustrate the concepts.
\end{abstract}



\section{Introduction}

Fractional order systems have been used in a wide range of applications such as electrical circuits, electromagnetism, viscoelasticity, neuroscience, and electrochemistry \cite{debnath2003recent,hilfer2000applications,kilbas2006new}.
There exists extensive research on studying dynamical properties of these systems. For instance, stability of the fractional order systems has been covered extensively (see, e.g. the survey \cite{li2011survey} and the references therein). In particular, stability of linear fractional order systems \cite{ahn2008necessary,deng2007stability,lu2009robust}, interval linear fractional order systems \cite{ahn2008necessary}, linear fractional order systems with multiple time delays \cite{deng2007stability}, and nonlinear fractional order systems \cite{Li2009,Li2010, sadati2010mittag,lenka2016asymptotic} has been studied. Robust stability using linear matrix inequalities (LMI) has also been studied in \cite{lu2009robust}.  An interesting generalization of Lyapunov stability for nonlinear fractional order systems is Mittag-Leffler stability that has been studied in works such as \cite{Li2009,Li2010, sadati2010mittag}.\par
In this paper, we are interested in extending the tools of passivity and dissipativity for fractional order systems. Passivity, and more generally dissipativity, are powerful tools which propose an energy-based method for analyzing dynamic systems~\cite{antsaklis2013control}. While initiating from concepts in electrical circuits; these tools have been extended to a wide range of applications, such as multi-agent systems \cite{chopra2006passivity}, cyber-physical systems \cite{antsaklis2013control}, chemical and thermodynamics processes \cite{bao2007process}, and aerospace structures \cite{balas2012direct}. One reason for the popularity of these tools is that they guarantee other useful properties. For instance, a passive integer order system is also stable under mild conditions and passivity is preserved under feedback and parallel interconnections \cite{hill1976stability,sepulchre2012constructive}.
 In this paper, we propose suitable notions of passivity and dissipativity for fractional order systems and show that similar properties of stability and compositionality hold for fractional order systems as well. We also consider the problem of designing a feedback controller to passivate a linear fractional order system and show that similar to linear integer order system, we can solve for the controller as a linear matrix inequality (LMI).
\par
The rest of the paper is organized as follows; In Section \ref{sec:pre}, some basic definitions and results from fractional calculus and fractional order systems are presented. In Section \ref{sec:dissipativity}, we propose passivity and dissipativity notions for fractional order systems, and study their stability and compositionality implications. Section \ref{passivationsec} solves the feedback passivation problem for fractional order systems. In Section \ref{simulation}, some examples are presented to illustrate the proposed concepts. Finally, some future research directions and conclusion are presented.
\par
\textit{Notations:} We will use the following notations. $\Re$ denotes the set of reals, $\Re_{+}$ the set of non-negative reals, $\mathcal{Z}$ the set of integers. $\Re ^{n}$ denotes the set of $n$-dimensional real vector. The Laplace transform of a function $f(t)$ is denoted by $\mathcal{L}(f(t))$. Convolution of two functions $f$ and $g$ is denoted by $f*g$. $\underline{0}$ denotes the zero vector in the space of whose dimension will be clear from the context. $||.||$ denotes the 2-norm.
\section{Preliminaries}
\label{sec:pre}	
We begin by reviewing briefly some mathematical preliminaries.
\subsection{Gamma Function}
The gamma function is the generalization of the factorial operator and is defined for any $z\in \Re$ as \cite{Loverro2004}
\begin{equation}
\Gamma (z)=\int _{0}^{\infty} e^{-u}u^{z-1} du.
\end{equation}

\subsection{The Mittag-Leffler Function}
The Mittag-Leffler Function is one of the most important functions in the fractional calculus and is defined as \cite{Loverro2004}
\begin{equation}
E_{\alpha} (z)=\sum _{k=0}^{\infty} \frac{z^{k}}{\Gamma (\alpha k+1)}, \alpha >0, z\in \Re.
\label{mitt}
\end{equation}
As we can see, for $\alpha=1$, $E_{\alpha} (z)=e^{z}$.

\subsection{The Fractional Integral}
The fractional integral of the function $f(t)$ with order $\alpha \in \Re _{+}$ is given as \cite{Loverro2004}
\begin{equation}
J_{t}^{\alpha}f(t):= f_{\alpha}(t)=\frac{1}{\Gamma (\alpha)}\int _{0}^{t} (t-\tau )^{\alpha -1}f(\tau)d\tau.
\label{fracint}
\end{equation}
It can be proven that
\begin{enumerate}
\item $J_{t}^{0} f(t)=f(t)$.
\item $J_{t}^{\alpha}J_{t}^{\beta}f(t)=J_{t}^{\alpha +\beta}f(t)=J_{t}^{\beta}J_{t}^{\alpha}f(t), \forall \alpha, \beta \in \Re_{+}$.
\end{enumerate}
\subsection{The Fractional Derivative}
We will define a fractional order derivative through the Right-Hand Definition (RHD) or the Caputo derivative, which is defined as \cite{Loverro2004}
\begin{equation}
D_{t}^{\alpha}f(t):= \frac{1}{\Gamma(m-a)}\int _{0}^{t}\frac{f^{(m)}(\tau)}{(t-\tau)^{\alpha+1-m}}d\tau ,\quad m-1<\alpha <m, \alpha \in \Re^{+}, m\in \mathcal{Z}.
\end{equation}
\par
It may be noted that
\begin{enumerate}
\item $D^{\beta}=J^{1-\beta}D^{1}, \forall \beta \in \Re _{+}$.
\item $D^{\alpha}D^{\beta}f(t)=D^{\beta}D^{\alpha}f(t),\forall \alpha, \beta \in \Re_{+}$.
\item $\mathcal{L}(D^{\beta}_{t}f(t))=s^{\beta}\mathcal{L}(f(t))-s^{\beta -1}f(0)$.
\end{enumerate}
\subsection{Fractional order systems and Mittag-Leffler Stability}
Consider the following fractional order system
\begin{equation}
\begin{aligned}
&D_{t}^{\gamma}x(t)=f(x,u,t),\\
&y=h(x,u,t)
 \end{aligned}
\label{system}
\end{equation}
where $x(t)\in \Re ^{n}$, $y(t) \in \Re ^{m}$, $y(t) \in \Re^{p}$, $\gamma \in (0,1)$ is the order of the system, and $f$ is a piecewise continuous function in $t$ and locally Lipschitz on a set $X$ containing the origin $x=\underline{0}$.\par
\begin{definition} (Free system) The system \eqref{system} evolving with $u(t)\equiv \underline{0}, \forall t\geq 0$ is called the free system of \eqref{system}.
\end{definition} 
\begin{definition}\cite{Li2010}(Equilibrium point)The free system of \eqref{system} has an equilibrium point at $x_0$, if and only if $f(x_{0},\underline{0},t)=\underline{0}$.
\end{definition}
\begin{assumption}
Without loss of generality, we will assume in the paper that the equilibrium point is at the origin.
\end{assumption}
\begin{definition}
\label{MLstability}
\cite{Li2010} 
For the free system of \eqref{system}, the equilibrium point $x=0$ is \textit{Mittag-Leffler} stable, if for all $t\geq 0$, and a $\beta \in (0,1)$, there exist positive real constants $\alpha _{1}$, $\alpha _{2}$, $\alpha _{3}$, $a$, and $b$, and a sufficiently smooth function $V(x,t): [0,\infty)\times X\rightarrow \Re$, such that 
\begin{equation}
\alpha _{1}\parallel x \parallel ^{a}\leq V(x,t) \leq \alpha _{2}\parallel x \parallel ^{ab},
\label{first}
\end{equation}
and
\begin{equation}
D_{t}^{\beta}V(x,t)\leq -\alpha _{3} \parallel x\parallel^{ab}
\label{second}
\end{equation}
\end{definition}
\begin{proposition}
\label{sta}
Let the free system of \eqref{system} satisfy
\begin{equation}
D_{t}^{\beta}V(x,t)\leq 0,
\label{third}
\end{equation}
and
\begin{equation}
\alpha _{1}\parallel x \parallel ^{a}\leq V(x,t).
\label{positive}
\end{equation}
Then, the free system of \eqref{system} is stable in the sense of Lyapunov.
\end{proposition}
\begin{proof}
According to \eqref{third}, there exists a non-negative function $M(t)$, such that
\begin{equation}
D_{t}^{\beta}V(x,t)+M(t)= 0.
\end{equation}
Taking the Laplace transform yields
\begin{equation}
s^{\beta}V(s)-s^{\beta -1}V(0)+M(s)=0,
\end{equation}
where $V(0) := V(x(0),0)$, and $V(s) := \mathcal{L} \lbrace V(x,t) \rbrace$. Equivalently,
\begin{equation}
V(s)=\frac{V(0)}{s}-\frac{M(s)}{s^{\beta}}.
\end{equation}
Performing the inverse Laplace transform yields
\begin{equation}
V(x,t)=V(x(0),0)-M(t)\ast \frac{t^{\beta -1}}{\Gamma (\beta)}.
\label{forth}
\end{equation}
Since $M(t)$ and $\frac{t^{\beta -1}}{\Gamma (\beta)}$ are positive functions for all $t$, and $\beta$, \eqref{forth} implies
\begin{equation}
V(x,t)<V(x(0),0).
\label{fifth}
\end{equation}
From (\ref{fifth}) and (\ref{positive}), we can infer that
\begin{equation}
\alpha _{1}\parallel x \parallel ^{a}<V(x(0),0).
\end{equation}
Equivalently, we can write
\begin{equation}
\parallel x \parallel <\left(\frac{V(x(0),0)}{\alpha _{1}}\right)^{\frac{1}{a}}.
\label{stable}
\end{equation}
Equation (\ref{stable}) shows that origin is a stable equilibrium point for the free system of~(\ref{system}). Furthermore, we can assume that there exists an $\epsilon$, such that
\begin{equation}
\parallel x \parallel <\left(\frac{V(0)}{\alpha _{1}}\right)^{\frac{1}{a}}< \epsilon.
\label{sta1}
\end{equation}
Then, according to (\ref{positive}) and \eqref{stable}, there exist positive constants $\alpha _{n}$ and $n$ such that $V(x(0),0)$ can be written as
\begin{equation*}
V(x(0),0)=\alpha _{n}\parallel x(0) \parallel ^{n},
\end{equation*}
which after some simplification, yields
\begin{equation}
\parallel x(0) \parallel < \left(\frac{\alpha _{1}}{\alpha _{n}} \epsilon ^{a}\right)^{\frac{1}{n}}=\theta.
\end{equation}
Therefore, the system (\ref{system}) is stable in the sense of Lyapunov.
\end{proof}
\begin{remark}
Any fractional order system with an order of derivative of more than $1$ can be transformed into an augmented system with the derivative order of less than $1$. To do so, consider the system in \eqref{system} where $\beta >1$. Without loss of generality, assume that $\frac{\beta}{2}<1$. Then, we can rewrite the system as
\begin{equation}
D_{t}^{\frac{\beta}{2}}(D_{t}^{\frac{\beta}{2}}x(t))=f(x,u,t).
\label{systemnew}
\end{equation}
Then, defining
\begin{equation}
D_{t}^{\frac{\beta}{2}}x(t)=z(t),
\end{equation}
yields the following augmented system
\begin{equation}
D_{t}^{\frac{\beta}{2}}\begin{bmatrix}
x(t)^{T} &z(t)^{T}
\end{bmatrix}^{T}=\begin{bmatrix}
f(x,u,t)^{T}&x(t)^{T}
\end{bmatrix}^{T}.
\end{equation}
\end{remark}
\section{Dissipativity of fractional order systems}
\label{sec:dissipativity}
\begin{definition}(Dissipativity)
The system \eqref{system} is dissipative with respect to the supply rate $w(x,u,y)$ and with order of $\beta \in (0,1)$, if there exists a positive semi-definite storage function $V(x,t)$ such that
\begin{equation}
D_{t}^{\beta}V(x,t)\leq w(x,u,y), \quad \forall t\geq 0.
\label{diss}
\end{equation}
\end{definition}
\begin{definition}(QSR dissipativity)
The system \eqref{system} is QSR dissipative with order of $\beta$, if for the supply rate
\begin{equation}
w(u,y)=-y^{T}Qy+y^{T}Su+u^{T}Ru
\label{QSR}
\end{equation}
where $Q$, $S$, and $R$ are matrices with proper dimensions, the dissipativity inequality~\eqref{diss} holds for a $\beta \in (0,1)$.
\end{definition}
\begin{theorem}
(Dissipativity and $L_2$ stability) Assume that the system \eqref{system} is QSR dissipative with $Q>0$ and a $\beta \in (0,1)$. Then, the system is $L_2$ stable.
\end{theorem}
\begin{proof}
Define $q=\parallel Q \parallel \geq 0$, $r=\parallel R \parallel \geq 0$, and $s=\parallel S \parallel > 0$. Then, according to \eqref{diss} and \eqref{QSR}, we obtain
\begin{equation}
D_{t}^{\beta}V(x,t)\leq -q \parallel y \parallel ^{2} + r \parallel u \parallel ^{2}+s \parallel u \parallel  \parallel y \parallel 
\label{near24}
\end{equation}
Then, adding and subtracting $\frac{s^2}{2q}\parallel u \parallel^{2}$ to the right side of \eqref{near24} yields 
\begin{equation}
D_{t}^{\beta}V(x,t)\leq -\frac{1}{2q}(q\parallel y \parallel -s\parallel u \parallel )^{2}+(\frac{s^2}{2q}+r)\parallel u \parallel ^{2}-\frac{q}{2}\parallel y \parallel ^{2}.
\label{l2mid}
\end{equation}
Therefore, we can rewrite \eqref{l2mid} as
\begin{equation}
D_{t}^{\beta}V(x,t)\leq (\frac{s^2}{2q}+r)\parallel u \parallel ^{2}-\frac{q}{2}\parallel y \parallel ^{2}.
\label{secl2}
\end{equation}
Equivalently, we can write
\begin{equation}
J_{t}^{1-\beta}D_{t}^{1}V(x,t)\leq (\frac{s^2}{2q}+r)\parallel u \parallel ^{2}-\frac{q}{2}\parallel y \parallel ^{2}.
\end{equation}
Taking the integer order normal integral yields
\begin{equation}
J_{t}^{1-\beta}J_{\tau}^{1}D_{t}^{1}V(x,t)\leq (\frac{s^2}{2q}+r)\parallel u_{\tau} \parallel _{L_{2}}^{2}-\frac{q}{2}\parallel y_{\tau} \parallel _{L_{2}}^{2}.
\label{secl2one}
\end{equation}
where $u_{\tau}$, and $y_{\tau}$ are the truncated version of $u$, and $y$ until time $\tau$.
Then, we rewrite \eqref{secl2one} as
\begin{equation}
J_{t}^{1-\beta}(V(x,\tau)-V(x,0))\leq (\frac{s^2}{2q}+r)\parallel u_{\tau} \parallel _{L_{2}}^{2}-\frac{q}{2}\parallel y_{\tau} \parallel _{L_2}^{2}.
\label{secl2two}
\end{equation}
Then, since $V(x,t)\geq 0, \forall t$, the fractional order integral of $V(x,\tau)$ is also positive. Furthermore, according to the definition of the fractional order integral \eqref{fracint}, we can derive
 \begin{equation}
\frac{q}{2}\parallel y_{\tau} \parallel _{L_2}^{2} \leq (\frac{s^2}{2q}+r)\parallel u_{\tau} \parallel _{L_{2}}^{2}+\frac{t ^{1-\beta}}{(1-\beta)\Gamma (1-\beta)}V(x,0), \quad \forall t \in [0,\tau].
\label{secl2three}
\end{equation}
Since for every $\tau$ \eqref{secl2three} should hold for all $t \in [0,\tau]$, it should also hold for $t=0$, i.e.,
 \begin{equation}
\frac{q}{2}\parallel y_{\tau} \parallel _{L_2}^{2} \leq (\frac{s^2}{2q}+r)\parallel u_{\tau} \parallel _{L_{2}}^{2}, \quad \forall \tau\geq 0,
\label{secl2four}
\end{equation}
which shows the system is $L_2$ stable.
\end{proof}
\begin{remark}
As $\beta \rightarrow 1$, since $D_{t}^{\beta} \rightarrow \frac{d}{dt}$, the fractional order system converges to an integer order system. In this case,
\begin{equation*}
\lim _{\beta \rightarrow 1}\frac{t^{1-\beta}}{(1-\beta)\Gamma(1-\beta)}=1
\end{equation*}
Therefore, one can rewrite \eqref{secl2three} as
 \begin{equation}
\frac{q}{2}\parallel y_{\tau} \parallel _{L_2}^{2} \leq (\frac{s^2}{2q}+r)\parallel u_{\tau} \parallel _{L_{2}}^{2}+V(x,0), \quad \forall t \in [0,\tau],
\end{equation}
which recovers the result that QSR dissipativity of integer order systems implies $L_2$ stability.
\end{remark}
\begin{definition} (Passivity of fractional order systems)
The system (\ref{system}) is called passive, if there exists a positive semi-definite storage function function $V(x,t)$, such that (\ref{positive}) is dissipative with order of $\beta \in (0,1)$ with respect to the supply rate $w(u,y)=u^{T}y$, i.e.,
\begin{equation}
u^{T}y \geq  D_{t}^{\beta}V(x,t).
\label{passivity}
\end{equation}
\end{definition}

\begin{theorem} (Passivity and Lyapunov stability)
If the system (\ref{system}) is passive with a positive semi-definite storage function $V(x,t)$, then $D_{t}^{\gamma}x(t)=f(x,0,t)$ is stable in the sense of Lyapunov with origin as the equilibrium point.\par
\end{theorem}
\begin{proof}
Setting $u=0$ in (\ref{passivity}) results in the equation (\ref{third}), which is proven in \textit{Proposition} \ref{sta} to yield the conclusion that the system is stable in the sense of Lyapunov.
\end{proof}
\begin{definition}(State strictly passive)
The system (\ref{system}) is called  state strictly passive, if there exists a function $V(x,t)$ such that (\ref{positive}) is satisfied, and further for some positive definite functions $\psi (x)$
\begin{equation}
u^{T}y \geq D_{t}^{\beta}V(x,t)+\psi (x).
\label{strictpassivity}
\end{equation}
\end{definition}
\begin{theorem}
If the system (\ref{system}) is strictly passive with a positive storage function $V(x,t)$ satisfying (\ref{first}), and $\psi (x)=\alpha _{3} \parallel x\parallel^{ab}$, then the origin of $D_{t}^{\gamma}x(t)=f(x,0,t)$ is \textit{Mittag-Leffler} stable.\par
\end{theorem}
\begin{proof}
Setting $u=\underline{0}$ in (\ref{strictpassivity}) results in (\ref{second}), which according to \textit{Definition} \ref{MLstability}, it satisfies the \textit{Mittag-Leffler} stablity.
\end{proof}
\begin{theorem}
Assume two systems $S_1$ and $S_2$ with the following dynamics
\begin{equation*}
\begin{aligned}
{S_{1}:} &D_{t}^{\gamma _{1}}x_{1}(t)=f_{1}(x_{1},u_{1},t),\\& y_{1}=h_{1}(x_{1},u_{1},t),
\end{aligned}
\end{equation*}
and,
\begin{equation*}
\begin{aligned}
{S_{2}:} &D_{t}^{\gamma _{2}}x_{2}(t)=f_{2}(x_{2},u_{2},t),\\& y_{2}=h_{2}(x_{2},u_{2},t).
\end{aligned}
\end{equation*}
If both systems are passive with the same order of $\beta$, then the parallel and feedback interconnections of $S_1$, and $S_2$ are also passive with the order of $\beta$.
\end{theorem}
\begin{proof}
\textit{Parallel interconnection:} Consider the parallel interconnection in Fig. \ref{parallel}. Since both systems are passive with the same order of $\beta$, there exist two functions $V_{1}$ and $V_{2}$ such that the passivity equation \eqref{passivity} holds for both systems. i.e.,
\begin{equation}
u_{1}^{T}y_{1} \geq  D_{t}^{\beta}V_{1}(t,x_{1}),
\label{sys1}
\end{equation}
and,
\begin{equation}
u_{2}^{T}y_{2} \geq  D_{t}^{\beta}V_{2}(t,x_{2}).
\label{sys2}
\end{equation}
Since in the parallel interconnection $u=u_{1}=u_{2}$, and $y=y_{1}+y_{2}$, adding \eqref{sys1}, and \eqref{sys2} yields
\begin{equation}
u^{T}y_{1}+u^{T}y_{2}\geq D_{t}^{\beta}(V_{1}(t,x_{1})+V_{2}(t,x_{2})).
\end{equation}
Then, defining $x=\begin{bmatrix}
x_{1}^{T} & x_{2}^{T}
\end{bmatrix} ^{T}$, and $V(x,t)=V_{1}(x_{1},t)+V_{2}(x_{2},t)$ yields
\begin{equation}
u^{T}y \geq D_{t}^{\beta}V(x,t),
\end{equation}
which shows that the parallel interconnected system is also passive with the order of $\beta$.

\begin{figure}[htbp]
  \centering
  \includegraphics[scale=0.4]{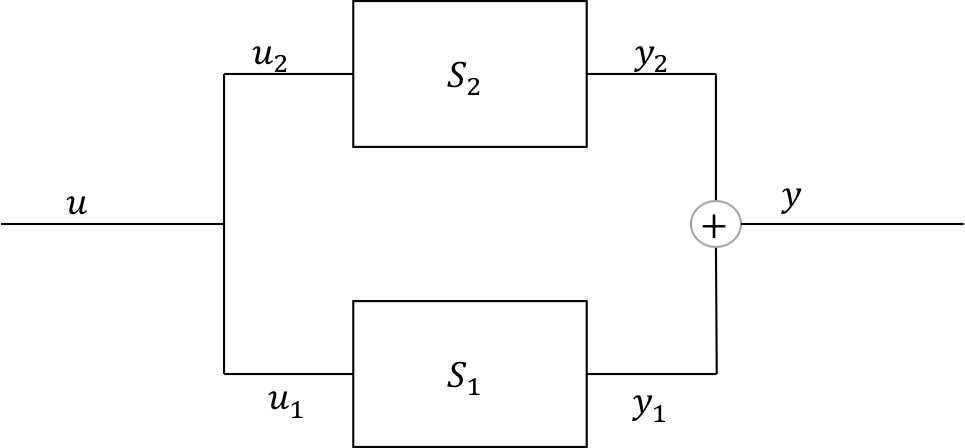}
  \caption{Parallel interconnection of two passive systems}
  \label{parallel}
\end{figure}
\par
\textit{Feedback interconnection:} Consider the feedback interconnection in Fig. \ref{feedback}. In this case $u_{2}=y_{1}$, and $u_{1}=r-y_{2}$. Therefore, adding \eqref{sys1}, and \eqref{sys2} yields
\begin{equation}
u_{1}^{T}y_{1}+u_{2}^{T}y_{2} =(r-y_{2})^{T}y_{1}+y_{1}^{T}y_{2}\geq D_{t}^{\beta}(V_{1}(x_{1},t)+V_{2}(x_{2},t)),
\end{equation}
Then, with defining $x=\begin{bmatrix}
x_{1}^{T} & x_{2}^{T}
\end{bmatrix} ^{T}$, and $V(x,t)=V_{1}(x_{1},t)+V_{2}(x_{2},t)$, we derive
\begin{equation}
r^{T}y\geq D_{t}^{\beta}V(x,t),
\end{equation}
which shows that the closed-loop system is also passive.
\end{proof}
\begin{figure}[htbp]
  \centering
  \includegraphics[scale=0.4]{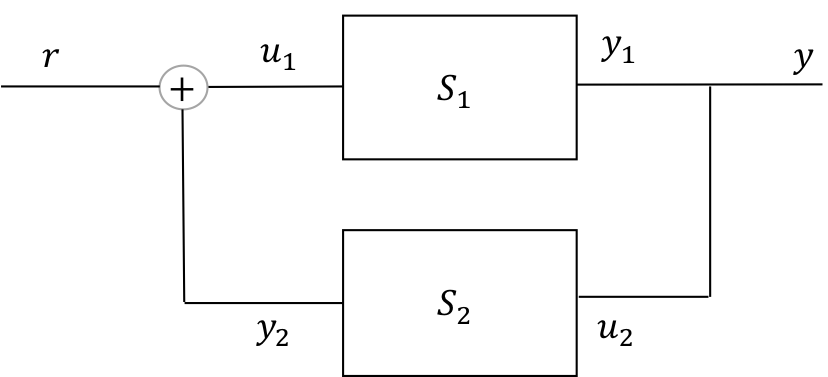}
  \caption{Feedback interconnection of two passive systems}
  \label{feedback}
\end{figure}
\par
The aforementioned theorem can be used when both systems are passive with the same order of $\beta$. In the following, we study the passivity of the interconnected systems with different orders of passivity. To do so, we need to use an equivalent definition for passivity which is described in the following proposition.
\par
\begin{proposition}
The system \eqref{system} is passive, if for any $\tau \geq 0$, the following equation holds
\begin{equation}
J_{\tau}^{1} u^{T}y \geq 0.
\label{newpassivity}
\end{equation}
\end{proposition}
\begin{proof}
According to \eqref{passivity}, we can write
\begin{equation}
u^{T}y\geq J_{t}^{1-\beta}D_{t}^{1}V(x,t).
\end{equation}
Taking integer order integral on both sides yields
\begin{equation}
J_{\tau}^{1}u^{T}y\geq J_{t}^{1-\beta}J_{\tau}^{1}D_{t}^{1}V(x,t).
\end{equation}
Equivalently, we derive
\begin{equation}
J_{\tau}^{1}u^{T}y\geq J_{t}^{1-\beta}(V(x,\tau)-V(x,0)).
\label{inequ1}
\end{equation}
Then, since $V(x,t)\geq 0, \forall t\geq 0$, \eqref{inequ1} can be written as
\begin{equation}
J_{\tau}^{1}u^{T}y\geq -\frac{t ^{1-\beta}}{(1-\beta)\Gamma (1-\beta)}V(x,0),\quad \forall 0 \leq t \leq \tau.
\label{inequ2}
\end{equation}
Since for every $\tau$, \eqref{inequ2} should holds for all $t \in [0,\tau ]$, it should also hold for $t=0$, which means
\begin{equation}
J_{\tau}^{1}u^{T}y\geq 0.
\end{equation}
\end{proof}
\begin{remark}
For the case $\beta \rightarrow 1$, \eqref{inequ2} becomes
\begin{equation}
J_{\tau}^{1}u^{T}y\geq -V(x,0),
\end{equation}
which is the definition of passivity in integer order systems.
\end{remark}
\begin{theorem}
Assume two passive fractional order systems $S_1$ and $S_2$ with passivity orders of $\beta _{1}$ and $\beta _{2}$, respectively. Furthermore, $V_{1}(x_{1},0)=V_{2}(x_{2},0)=0$. Then, the overall system with the parallel and feedback interconnection is passive.
\end{theorem}
\begin{proof}
Assume that according to \eqref{newpassivity}, for both systems $S_{1}$ and $S_2$, the passivity condition holds as follows
\begin{equation*}
J_{\tau}^{1} u_{1}^{T}y_{1} \geq 0,
\end{equation*}
\begin{equation}
J_{\tau}^{1} u_{2}^{T}y_{2} \geq 0.
\label{twosys}
\end{equation}
Therefore, for the parallel interconnection, we can write
\begin{equation}
J_{\tau}^{1} u^{T}y_{1}+J_{\tau}^{1} u^{T}y_{2}\geq 0.
\label{newparallel}
\end{equation}
Equivalently,
\begin{equation}
J_{\tau}^{1}(u^{T}y_{1}+u^{T}y_{2})\geq 0.
\end{equation}
Showing the system with the parallel interconnection is also passive.
\par
A similar approach can be used to show that the feedback interconnection of two passive fractional order systems is also passive. For the feedback interconnection, since $u_{1}=r-y_{2}$, and $u_{2}=y_{1}$, adding equations in \eqref{twosys} yields 
\begin{equation}
J^{1} (r-y_{2})^{T}y_{1}+J^{1} {y_{1}}^{T}y_{2}\geq 0.
\end{equation}
Equivalently, we write
\begin{equation}
J_{\tau}^{1} r^{T}y_{1} \geq 0,
\end{equation}
which shows the closed-loop system is passive.\par
\end{proof}
For the sake of completeness, in following we review the input feed-forward output feedback passive (IF-OFP), and the passivity indices for the fractional order systems.
\begin{definition}(IF-OFP) \cite{khalil1996noninear,sepulchre2012constructive}
System \eqref{system} is input feed-forward output feedback passive (IF-OFP), if there exist $\rho , \nu \in \Re $ such that the system is dissipative with respect to the following supply rate
\begin{equation}
w(u,y)=u^{T}y-\nu u^{T}u -\rho y^{T}y, \forall t \geq 0.
\label{IFE}
\end{equation}
\label{IF}
\end{definition}
To describe how much a system is passive, passivity indices are defined. \par
\begin{definition}(Passivity indices) \cite{khalil1996noninear,sepulchre2012constructive}
According to the \textit{Definition} \ref{IF}, a system is
\begin{itemize}
\item input feed-forward (strictly) passive (IFP), if the system \eqref{system} is dissipative with $\rho =0$, and $\nu \neq 0$ ($\nu >0$) in \eqref{IFE},
\item output feedback (strictly) passive (OFP), if the system \eqref{system} is dissipative with $\rho \neq 0$ ($\rho > 0$), and $\nu = 0$ ($\nu >0$) in \eqref{IFE}, and
\item very strictly passive (VSP), if the system \eqref{system} is dissipative with $\rho > 0$, and $\nu > 0$ in \eqref{IFE}.
\end{itemize}
\end{definition}
\section{Feedback passivation}
\label{passivationsec}
In this section, the problem of feedback passivation is studied for linear fractional order systems. In this problem, we want to design a controller such that the closed-loop system is passive for all the external inputs.\par
Consider the following linear fractional order system
\begin{equation}
\begin{aligned}
&D_{t}^{\gamma}x=Ax+Bu,\\&y=Cx+Du.
\end{aligned}
\label{linear}
\end{equation}
The objective is to design a controller with the following form
\begin{equation}
u=Fx+Gv,
\label{controller}
\end{equation}
such that the closed-loop system is passive for all $v$, where $v$ is an external input.\par
In following, we first review a lemma which is useful to prevent from using the chain rule of derivatives of fractional order. Then, we will propose the feedback passivation theorem for linear fractional order systems.
\begin{lemma}\cite{zhao2016state}
For all time instants $t$, there exist a positive definite matrix $P$ such that the following inequality holds
\begin{equation}
\frac{1}{2}D_{t}^{\beta}(x^{T}Px)\leq x^{T}PD_{t}^{\beta}x.
\label{lem}
\end{equation}
\end{lemma}
The following theorem proposes a control design approach for the feedback passivation problem.
\begin{theorem}(Feedback passivation)
Consider the linear system \eqref{linear}, and the controller \eqref{controller}. The closed-loop system is passive with the storage function $V(x)=\frac{1}{2}x^{T}Px$, if there exist a matrix $X=P^{-1}$ and a matrix $Q=FX$ with proper dimensions, such that

\begin{equation}
\begin{bmatrix}
    (AX+XA^{T}+BQ+Q^{T}B^{T})       & (BG-XC^{T}-Q^{T}D^{T}) \\
   (BG-XC^{T}-Q^{T}D^{T})^{T}       & (-DG-G^{T}D^{T})
\end{bmatrix}
\leq 0.
\end{equation}
Then, using $F=QX^{-1}$, the controller can be designed.
\label{thepassivation}
\end{theorem}
\begin{proof}
Consider the system \eqref{linear} and the controller \eqref{controller}, we can write the closed-loop system as
\begin{equation}
\begin{aligned}
&D_{t}^{\gamma}x=(A+BF)x+BGv,\\&y=(C+DF)x+DGv.
\end{aligned}
\end{equation}
Recalling the passivity definition \eqref{passivity}, for the closed-loop system, we can write the passivity condition to be
\begin{equation}
D_{t}^{\beta}V\leq v^{T}(C+DF)x+v^{T}DGv.
\end{equation}
Assume that the storage function candidate is $V=\frac{1}{2}x^{T}Px$, and $\beta=\gamma$. Then using~\eqref{lem}, if the following equation holds
\begin{equation}
D_{t}^{\beta}(\frac{1}{2}x^{T}Px)\leq x^{T}PD_{t}^{\beta}x=x^{T}P((A+BF)x+BGv)\leq v^{T}(C+DF)x+v^{T}DGv,
\end{equation}
the closed-loop system is passive. Equivalently, we write
\begin{equation}
x^{T}P(A+BF)x+x^{T}PBGv-v^{T}(C+DF)x-v^{T}DGv \leq 0.
\end{equation}
Now, since each term is scalar, we rewrite it as
\begin{equation}
\begin{aligned}
 \frac{1}{2} & (x^{T}PAx+x^{T}A^{T}Px+x^{T}PBFx+x^{T}F^{T}B^{T}Px+x^{T}PBGv+v^{T}G^{T}B^{T}Px \\
& -v^{T}Cx-x^{T}C^{T}v-v^{T}DFx-x^{T}F^{T}D^{T}v-v^{T}DGv-v^{T}G^{T}D^{T}v) \leq 0.
\end{aligned}
\end{equation}
Then, it yields
\begin{equation}
\begin{bmatrix}
x^{T} & v^{T}
\end{bmatrix}
\begin{bmatrix}
\frac{1}{2}(PA+A^{T}P+PBF+F^{T}B^{T}P)&\frac{1}{2}(PBG-C^{T}-F^{T}D^{T})\\
\frac{1}{2}(PBG-C^{T}-F^{T}D^{T})^{T}&\frac{1}{2}(-DG-G^{T}D^{T})
\end{bmatrix}
\begin{bmatrix}
x\\
v
\end{bmatrix}\leq 0,
\end{equation}
which implies that
\begin{equation}
\begin{bmatrix}
\frac{1}{2}(PA+A^{T}P+PBF+F^{T}B^{T}P)&\frac{1}{2}(PBG-C^{T}-F^{T}D^{T})\\
\frac{1}{2}(PBG-C^{T}-F^{T}D^{T})^{T}&\frac{1}{2}(-DG-G^{T}D^{T})
\end{bmatrix}\leq 0.
\label{firstineq}
\end{equation}
Multiplying both sides of the matrix \eqref{firstineq} by 
\begin{equation*}
\begin{bmatrix}
P^{-1}&0\\
0&I
\end{bmatrix}
\end{equation*}
where $I$ is the identity matrix with proper dimensions, yields
\begin{equation}
\begin{bmatrix}
\frac{1}{2}(AP^{-1}+P^{-1}A^{T}+BFP^{-1}+P^{-1}F^{T}B^{T})&\frac{1}{2}(BG-P^{-1}C^{T}-P^{-1}F^{T}D^{T})\\
\frac{1}{2}(BG-P^{-1}C^{T}-P^{-1}F^{T}D^{T})^{T}&\frac{1}{2}(-DG-G^{T}D^{T})
\end{bmatrix}\leq 0.
\label{secondineq}
\end{equation}
Defining $P^{-1}=X$, and $FX=Q$, and multiplying both sides of inequality by $2$ complete the proof.
\end{proof}
\section{Examples}
\label{simulation}
In this section, we present some examples to illustrate the proposed concepts. Furthermore, a simulation is presented for the feedback passivation problem.\par
\textbf{Example 1.} In this example, we want to check the passivity of the following system ($0<\beta <1$)
\begin{equation}
\begin{aligned}
S_{1}: &D_{t}^{\beta} z=u_{1},\\&y_{1}=z.
\end{aligned}
\label{ex1}
\end{equation}
Consider the storage function candidate $V_{1}(z)=\frac{1}{2}z^{2}$, then according to \eqref{passivity}, and \eqref{lem}, first we calculate $D_{t}^{\beta} V_{1}(z)$,
\begin{equation}
D_{t}^{\beta} V_{1}(z)=D_{t}^{\beta} (\frac{1}{2}z^2)\leq zD_{t}^{\beta} z=zu_{1}
\end{equation}
On the other hand, according to \eqref{ex1}, $u^{T}y_{1}=uz$. Therefore, \eqref{passivity} hold and the system \eqref{ex1} is passive with order of $\beta$.\par
\textbf{Example 2.} In this example, we are interested to check the passivity of the following system
\begin{equation*}
\begin{cases}
D_{t}^{\beta} x_{1}=x_{2}\\
D_{t}^{\beta} x_{2}=-ax_{1}-kx_{2}+u_{2}
\end{cases}
\end{equation*}
and,
\begin{equation}
y_{2}=x_{2}, \quad \forall a,k>0
\label{ex2}
\end{equation}
Consider the storage function $V_{2}(x)=\frac{1}{2}ax_{1}^{2}+\frac{1}{2}x_{2}^{2}$. Then, according to \eqref{passivity}, and \eqref{lem}, 
\begin{equation}
\begin{aligned}
D_{t}^{\beta} V_{2}(x)& =D_{t}^{\beta} (\frac{1}{2}ax_{1}^{2}+\frac{1}{2}x_{2}^{2})\leq ax_{1}D_{t}^{\beta}x_{1}+x_{2}D_{t}^{\beta}x_{2}\\
&=ax_{1}x_{2}+x_{2}(-ax_{1}-kx_{2}+u_{2})\\
&=-kx_{2}^{2}+x_{2}u_{2}
\end{aligned}
\end{equation}
which is less that $u^{T}y_{2}=ux_{2}$. Therefore, the system \eqref{ex2} is passive with order of $\beta$.\par 
\textbf{Example 3.} In this example, we study the compositionality of passive systems for the parallel and feedback interconnections. Consider two passive systems of \eqref{ex1}, and \eqref{ex2}.\\
\textit{Parallel interconnection.} In this case $y=y_{1}+y_{2}=z+x_{2}$, and $u=u_{1}=u_{2}$. Then, using the storage function candidate $V(x,z)=V_{1}(z)+V_{2}(x)=\frac{1}{2}ax_{1}^{2}+\frac{1}{2}x_{2}^{2}+\frac{1}{2}z^{2}$, we derive
\begin{equation}
\begin{aligned}
D_{t}^{\beta} V(x,z)=D_{t}^{\beta} (V_{1}(z)+V_{2}(z))& \leq ax_{1}D_{t}^{\beta}x_{1}+x_{2}D_{t}^{\beta}x_{2}+zD_{t}^{\beta}z\\
&=ax_{1}x_{2}+x_{2}(-ax_{1}-kx_{2}+u)+zu\\
&=-kx_{2}^{2}+x_{2}u+zu,
\end{aligned}
\end{equation}
which shows that considering $k>0$, and $u^{T}y=u(x_{2}+z)=ux_{2}+uz$, the overall system satisfy the passivity condition \eqref{passivity}. Therefore, the parallel interconnected system is also passive.\\
\textit{Feedback interconnection.} In this case, as denoted in Fig. \ref{feedback}, we know that $y=y_{1}=z=u_{2}$. Then, for the storage function $V(x,z)=V_{1}(x)+V_{2}(z)$,
\begin{equation}
\begin{aligned}
D_{t}^{\beta} V(x,z)& \leq ax_{1}D_{t}^{\beta}x_{1}+x_{2}D_{t}^{\beta}x_{2}+zD_{t}^{\beta}z\\
&=ax_{1}(x_{2})+x_{2}(-ax_{1}-kx_{2}+u_{2})+zu_{1}\\
&=-kx_{2}^{2}+x_{2}y+z(r-y_{2})\\
&=-kx_{2}^{2}+zr
\end{aligned}
\end{equation}
which satisfies the \eqref{passivity} for $u^{T}y=rz$. Therefore, the closed-loop system is also passive with order of $\beta$.\\
\textbf{Example 4.} In this example, the objective is to find a feedback controller of the form \eqref{controller} to passivate a  system which does not satisfy the passivity condition. Consider the following system,
\begin{equation*}
\begin{cases}
D_{t}^{\beta} x_{1}=x_{2}\\
D_{t}^{\beta} x_{2}=-ax_{1}+kx_{2}+u,
\end{cases}
\end{equation*}
and
\begin{equation}
y=x_{2}+u, \quad \forall a,k>0.
\label{ex4}
\end{equation}
It is easy to show that the passivity constraint \eqref{passivity} does not hold for all arbitrary $k>0$. Therefore, using the theorem \ref{thepassivation}, and solving the LMI to find  a feasible solution for the matrix $P$ in the storage function candidate $V(x)=\frac{1}{2}x^{T}Px$, we get 
\begin{equation}
P=\begin{bmatrix}
0.1617&0.0578\\
0.0578&0.0996
\end{bmatrix}.
\end{equation}
Therefore, the storage function $V(x)=0.1617x_{1}^{2}+0.1156x_{1}x_{2}+0.0996x_{2}^{2}$ satisfies the passivity constraint \eqref{passivity}. In this case, $F=\begin{bmatrix}
-0.5108&-0.3430
\end{bmatrix}$, and $G=5.6432$ are the designed feedback parameters to passivate the system \eqref{ex4}.

\section{Conslusion}
\label{conslusion}
In this paper, we presented the definitions of passivity, and dissipativity for the fractional order systems. Then, the properties of compositionality, and stability for the passive, and dissipative fractional order systems are proposed. It is shown that the free system of a passive fractional order system is Lyapunov stable.  Furthermore, the parallel and feedback interconnection of the passive systems preserve the passivity property. Moreover, we showed that QSR dissipative fractional order systems have the $L_2$ stability property. Then, we studied the feedback passivation problem in order to passivate a fractional order system through designing a feedback controller. Finally, we illustrated the concepts and theorems of the paper using presenting some examples.

\bibliographystyle{siamplain}
\bibliography{passivityref}
\end{document}

%% file: ex_shared.tex
\usepackage{lipsum}
\usepackage{amsfonts}
\usepackage{graphicx}
\usepackage{epstopdf}
\usepackage{algorithmic}
\usepackage{lineno}
\ifpdf
  \DeclareGraphicsExtensions{.eps,.pdf,.png,.jpg}
\else
  \DeclareGraphicsExtensions{.eps}
\fi

\numberwithin{theorem}{section}

\newcommand{\TheTitle}{On Passivity of Fractional Order Systems} 
\newcommand{\TheAuthors}{M. Rakhshan, V. Gupta, and B. Goodwine}

\headers{\TheTitle}{\TheAuthors}

\title{{\TheTitle}\thanks{Submitted to the editors DATE.
\funding{This work was funded by X}}}

\author{
  Mohsen Rakhshan\thanks{University of Notre Dame, Notre Dame, IN
    (\email{mrakhsha@nd.edu}).}
  \and
  Vijay Gupta\thanks{University of Notre Dame, Notre Dame, IN (\email{vgupta2@nd.edu}).}
  \and
  Bill Goodwine \thanks{University of Notre Dame, Notre Dame, IN (\email{billgoodwine@nd.edu}).}
}

\usepackage{amsopn}
